\documentclass[11pt, a4paper, reqno]{amsart}

\usepackage{pdflscape} 
\usepackage{hhline} 
\usepackage{array} 

\usepackage{calrsfs}
\DeclareMathAlphabet{\pazocal}{OMS}{zplm}{m}{n}

\usepackage{latexsym,amsmath,amsfonts,amscd,amssymb, mathrsfs, amsthm}
\usepackage[english]{babel}
\usepackage{appendix}
\usepackage{hyperref}
\usepackage{url}
\usepackage{mathtools}
\usepackage{booktabs}
\usepackage{tikz}
\usepackage{adjustbox}
\usepackage{float}

\advance\oddsidemargin by -0.5cm
\advance\evensidemargin by -0.5cm
\advance\topmargin by -1.3cm
\theoremstyle{plain}  
\newtheorem{theorem}{Theorem}[section]

\newtheorem*{theorem*}{Theorem}

\newtheorem{lemma}[theorem]{Lemma}
\newtheorem{proposition}[theorem]{Proposition}

\theoremstyle{definition}

\newtheorem{remark}[theorem]{Remark}

\newtheorem*{claim*}{Claim}

\numberwithin{equation}{section}

\newcommand{\R}{\mathbb{R}}

\newcommand{\Z}{\mathbb{Z}}

\newcommand{\C}{\mathbb{C}}

\newcommand{\g}{\mathfrak{g}}
\newcommand{\h}{\mathfrak{h}}
\newcommand{\z}{\mathfrak{z}}

\newcommand{\Aut}{\mathrm{Aut}}

\newcommand{\spann}{\mathrm{span}}

\usepackage{xcolor}

\newcommand{\be}{\begin{equation}}
	\newcommand{\ee}{\end{equation}}

\newcommand{\ben}{\begin{enumerate}}
	\newcommand{\een}{\end{enumerate}}
\newcommand{\bit}{\begin{itemize}}
	\newcommand{\eit}{\end{itemize}}
\newcommand{\edoc}{\end{document}}

\def\br#1\er{{#1}} 
\def\bw#1\ew{\textcolor{brown}{#1}} 
\def\bb#1\eb{\textcolor{blue}{#1}} 
\def\br#1\er{\textcolor{red}{#1}} 
\def\bm#1\em{\textcolor{magenta}{#1}}
\def\bv#1\ev{\textcolor{olive}{#1}}

\makeatletter
\renewcommand{\tocsection}[3]{%
	\indentlabel{\@ifnotempty{#2}{\ignorespaces#1 #2\quad}}#3}
\renewcommand{\tocsubsection}[3]{%
	\indentlabel{\@ifnotempty{#2}{\ignorespaces#1 #2\quad}}#3}

\newcommand\@dotsep{4.5}
\def\@tocline#1#2#3#4#5#6#7{\relax
	\ifnum #1>\c@tocdepth 
	\else
	\par \addpenalty\@secpenalty\addvspace{#2}%
	\begingroup \hyphenpenalty\@M
	\@ifempty{#4}{%
		\@tempdima\csname r@tocindent\number#1\endcsname\relax
	}{%
		\@tempdima#4\relax
	}%
	\parindent\z@ \leftskip#3\relax \advance\leftskip\@tempdima\relax
	\rightskip\@pnumwidth plus1em \parfillskip-\@pnumwidth
	#5\leavevmode\hskip-\@tempdima{#6}\nobreak
	\leaders\hbox{$\m@th\mkern \@dotsep mu\hbox{.}\mkern \@dotsep mu$}\hfill
	\nobreak
	\hbox to\@pnumwidth{\@tocpagenum{\ifnum#1=1\fi#7}}\par
	\nobreak
	\endgroup
	\fi}
\AtBeginDocument{%
\expandafter\renewcommand\csname r@tocindent0\endcsname{0pt}
}
\def\l@subsection{\@tocline{2}{0pt}{2.5pc}{5pc}{}}
\makeatother

\usepackage{csquotes}

\begin{document}
\title[The completeness problem on 3-dimensional non-unimodular Lie groups]{The completeness problem on 3-dimensional non-unimodular Lie
	groups}


\author[S. Chaib]{Salah Chaib}
\address{\hspace{-5mm} Salah Chaib, Centro de Matem\'{a}tica,
	Universidade do Minho,
	Campus de Gualtar,
	4710-057 Braga,
	Portugal} 
\email {salah.chaib@cmat.uminho.pt}
\author[A.C. Ferreira]{Ana Cristina Ferreira}
\address{\hspace{-5mm} Ana Cristina Ferreira, Centro de Matem\'{a}tica,
	Universidade do Minho,
	Campus de Gualtar,
	4710-057 Braga,
	Portugal} 
\email {anaferreira@math.uminho.pt}


\date{\today}

\bigskip

\begin{abstract}  
We consider the completeness problem for left-invariant Lorentzian metrics on 3-dimensional non-unimodular Lie groups, all of which have Lie algebra of the form $\mathbb{R} \ltimes_A \mathbb{R}^2$, where $A$ is a real $2 \times 2$ matrix with nonzero trace. The case where $A$ is not diagonalizable over $\mathbb{C}$ was addressed in previous work by the authors, and the limiting case where $A$ is a scalar multiple of the identity is also known from the literature. In this paper, we determine all geodesically (in)complete left-invariant Lorentzian metrics for all other cases where $A$ is diagonalizable over $\mathbb{R}$. Additionally, we show that, when $A$ is diagonalizable over $\mathbb{C}$ but not over $\mathbb{R}$, there exists at least one incomplete metric. As a consequence of prior work and our results, we obtain that every 3-dimensional non-unimodular Lie group admits an incomplete left-invariant Lorentzian metric.
\end{abstract}

\vspace*{-3mm}

\maketitle

\tableofcontents


\bigskip

\section{Introduction}

In this article, we address the following question: {\it Given a Lie group $G$ of dimension $3$, which of its left-invariant Lorentzian metrics are geodesically complete?}
Partial results about this question were previously obtained in the literature.  In \cite{BrombergMedina}, the completeness problem was solved for all 3-dimensional unimodular Lie groups. Hence, since any Lie group of dimension $3$ is either simple or solvable, and any simple Lie group is unimodular, we focus on 3-dimensional solvable (non-unimodular) Lie groups.

Any simply connected solvable Lie group $G$ is given by a semidirect product $\R\ltimes_A \R^2$, where $A$ is a $2\times 2$ real matrix and $\R$ acts on $\R^2$ via $t \mapsto \mathrm{exp}\, tA$. The isomorphism class of the Lie algebra $\g$ of $G$ is  thus determined by the Jordan normal form of $A$. We therefore have the following types of solvable Lie algebras, cf. \cite{Bianchi}, (i) $A$ is not diagonalizable over $\C$, (ii) $A$ is diagonalizable with real eigenvalues or (iii) $A$ is diagonalizable over $\C$ with complex non-real eigenvalues.      

There is, up to isomorphism, only one non-unimodular Lie algebra of type (i), the Lie algebra which we denote by $\mathfrak{psh}$. This case was considered in the first part of our study \cite{CFZ-partI}.  More precisely, we solved the completeness problem for the pseudo-homothetic Lie group, the simply connected Lie group whose Lie algebra is $\mathfrak{psh}$. In this manuscript, we investigate the case where $A$ is diagonalizable with real eigenvalues. Up to normalization with respect to the eigenvalue of highest absolute value, $A$ is similar to $\mathrm{diag}(1, \lambda)$, with $|\lambda| \leq 1$. We thus obtain a 1-parameter family of isomorphism classes of Lie algebras $\h(\lambda)$, $|\lambda|\leq 1$, such that the non-trivial bracket relations on $\h(\lambda)$ are given by
 \begin{equation}\label{eq:standbasis}
     [e_1, e_2] = e_2, \quad [e_1, e_3] = \lambda e_3.
 \end{equation}
Any basis $B=\{e_1, e_2,e_3\}$ satisfying \eqref{eq:standbasis} will be referred to as a standard basis of $\h(\lambda)$. 

The only unimodular Lie algebra in this family is obtained for the limiting case $\lambda = -1$ and corresponds to 
 the Lie algebra of the Poincar\'e
 group (also called $\mathrm{Sol}$), the group of isometries of the 2-dimensional Minkowski space.   Another limiting case occurs when $\lambda =1$ for the 3-dimensional homothetic group $\mathrm{Ho}$ given by $\R\ltimes_{\mathrm{Id}} \R^2$. It was proven in \cite{Guediri-solvable} that all left-invariant Lorentzian metrics on homothetic groups $\R\ltimes_{\mathrm{Id}} \R^{n}$ of any dimension are incomplete.  Moreover,  the  rather remarkable fact that {\it all} indefinite left-invariant metrics on $\R\ltimes_\mathrm{Id} \R^n$ are incomplete was later extended in \cite{VukmirovicSukilovic}.

Investigation of geodesic completeness of Lorentzian metrics on the remaining simply connected Lie groups with Lie algebra $\h(\lambda)$, $|\lambda|<1$, which includes the product $\mathrm{Aff}^+(\R)\times \R$ for $\lambda =0$, is the object of this part of our program. 

A classical and fundamental fact is that the study of the geodesic flow of a left-invariant metric on a Lie group $G$ can be reduced to the study of the flow of a certain vector field on $\g$, which we shall call the geodesic field, via the Euler-Arnold formalism \cite{Arnold-paper}. 
We present in Sec. \ref{sec:preliminaries} a brief account of background material and techniques on this topic.  
With this formalism in mind and for ease of exposition, we will speak of the geodesic completeness of a metric Lie algebra $(\g, q)$ to mean the geodesic completeness of $(G,g)$, where $G$ is a Lie group with Lie algebra $\g$ and $g$ is the left-invariant metric defined by $q$.

   In Sec. \ref{Sec:N-F}, using the fact that all Lie algebras $\h(\lambda)$, with $|\lambda|<1$, have the same automorphism group, we derive their metric normal forms, and present the list of their associated geodesic fields in Sec. \ref{Sec:geod-vf}. The special case $\lambda =0$, which corresponds to the Lie algebra $\mathfrak{aff}(\R)\oplus \R$, is considered in Sec. \ref{sec:(in)co_aff}, where we prove the following result.

\begin{theorem}\label{thm:aff}
Let $\mathfrak{aff}(\R)\oplus \R$ be equipped with a Lorentzian metric $q$, and $\z$ and $\mathfrak{d}$ be the center and the derived subalgebra, respectively. Then
  \begin{itemize}
  \item[(a)] if $\z$ is spacelike, q is incomplete with idempotents;
  \item[(b)] if $\z$ is timelike, q is complete and all integral curves of the geodesic field are bounded;
  \item[(c)] if $\z$ is lightlike,
  \begin{itemize}
      \item[(c.1)] and $\z$ and $\mathfrak{d}$ are not orthogonal, $q$ is incomplete with no idempotents;
      \item[(c.2)] and $\z$ and $\mathfrak{d}$ are orthogonal, q is complete and all non-stationary integral curves of the geodesic field are unbounded.
  \end{itemize}
  \end{itemize}
\end{theorem}

We proceed, in Sec. \ref{sec:(in)co_h_lambda}, with the study of geodesic completeness on the Lie algebras $\h(\lambda)$ for $0<|\lambda|<1$. Taking $B=\{e_1, e_2, e_3\}$ to be a standard basis of $\h(\lambda)$, we establish the following characterization. 

\begin{theorem}\label{thm:h(lambda)}
   Let $\h(\lambda)$, $0<|\lambda|<1$, be equipped with a Lorentzian metric $q$. Then, $q$ is geodesically complete if and only if $e_3$ is timelike and $e_2$ is not spacelike. Furthermore, if $e_3$ is timelike and
   \begin{itemize}
       \item[(a)] $e_2$ is timelike, all integral curves of the corresponding geodesic field are bounded; 
      \item[(b)]  $e_2$ is lightlike, unbounded curves of the corresponding geodesic field exist. 
      \end{itemize}
 \end{theorem}

 In Sec. \ref{sec:limiting}, we 
  study the limiting Lie algebras $\h(\lambda)$, $\lambda = \pm 1$. We recover the result of \cite{Guediri-solvable} for $\h(1)$ and clarify a misstatement in \cite[Prop. 3]{BrombergMedina} for $\h(-1)$. 

We finish this paper with Sec. \ref{sec:3D-non-uni}, by briefly considering the remaining family of non-unimodular Lie algebras of case (iii), in order to prove the following relevant result.

\begin{theorem}\label{Thm:Intro-Ic-3D}
    Every non-unimodular Lie algebra of dimension 3 admits an incomplete Lorentzian metric.
\end{theorem}


\section{Preliminaries}\label{sec:preliminaries}

In this section, we briefly recall the basic framework used throughout the text, fixing notation and terminology. An extended exposition can be found in \cite[Sec. 2]{CFZ-partI} and references therein. 

Let us fix a Lie group $G$ with Lie algebra $\g =T_eG$. Left-invariant metrics on $G$ are in one-to-one correspondence with non-degenerate symmetric bilinear forms on $\g$.  The Euler-Arnold theorem \cite{Arnold-paper} establishes that the geodesic flow of $(G,g)$ is uniquely associated with the flow of a vector field on $(\g, q=g_e)$ as follows.  If $\gamma$ is a geodesic with $\gamma(0)=e$, its velocity $\dot{\gamma}$ can be identified with a curve $v$ in $\mathfrak g$ via left translation. The geodesic equation then reduces to
\begin{equation}\label{eq:euler-arnold}
	\dot v(t) = \operatorname{ad}^\dagger_{v(t)} v(t),
\end{equation}
where $\operatorname{ad}^\dagger$ denotes the $q$-transpose of the adjoint map $\operatorname{ad}$. For simplicity of language, we will refer to the vector field defined by the system of ODEs \ref{eq:euler-arnold} as the geodesic field.  Clearly, a left-invariant metric on $G$ is geodesically complete if and only if its associated geodesic field on $\g$ is complete. 

 The geodesic field admits a natural first integral: the energy associated with $q$ is conserved along trajectories.  Further first integrals (eventually non-polynomial) may occur, providing effective constraints on the flow. For instance, if a positive-definite quadratic first integral exists then the flow is complete. On the other hand, a very useful tool for proving incompleteness is given by idempotents: an element $v_o \in \mathfrak g$ is called an idempotent if the geodesic field at $v_o$ is equal to $v_o$. The existence of an idempotent leads to an integral curve that blows up in finite time. However, not all incomplete trajectories arise from idempotents. The following lemma, \cite[Lemma 2.3]{CFZ-partI}, will be instrumental in providing such curves.

\begin{lemma}\label{lem:incompleteness}
	Let $(E)$ be an ordinary differential equation of the form
	$$
	\dot{x}(t)=ax^2(t)+\alpha(t),
	$$
	such that $a>0$ and $\alpha\in\mathcal{C}^\infty(\R)$; $t\mapsto\alpha(t)\geq0$. Let $\gamma\colon I\rightarrow\R$ be a nonzero maximal integral curve of $(E)$, then $\gamma$ must be incomplete.
\end{lemma}

Finally, another powerful tool is the automorphism group of the Lie algebra $\g$, $\operatorname{Aut}(\g)$. The automorphism group acts naturally on $\mathrm{Sym}^\ast(\g)$, the space of non-degenerate symmetric bilinear forms on $\g$.  
Not  surprisingly, completeness of the flow of the geodesic field of $(\g, q)$ is invariant under this action. This was proved, for instance, in \cite{ElshafeiFerreiraSanchezZeghib}. Moreover, the geodesic flow remains unchanged under rescaling. Therefore, all bilinear forms in each orbit of $\mathrm{Sym}^\ast(\g)$ by the action of $\R^\ast\times\mathrm{Aut}(\g)$ are either complete or incomplete. Representatives of such orbits are usually called metric normal forms.

\bigskip

\section{Metric normal forms on $\h(\lambda), |\lambda|< 1$}\label{Sec:N-F}

The classification of normal forms of left-invariant Lorentzian metrics in dimension 3 has been considered more or less implicitly in several articles. A detailed account was presented in \cite{HaLee23}, see also Appendix A in \cite{CFZ-isometry}. We include some computations here, for 3-dimensional Lie algebras $\h(\lambda)$ where $|\lambda|< 1$. Consider a standard basis $B=\{e_1, e_2, e_3\}$ satisfying the bracket relations
$$[e_1,e_2]=e_2,\ [e_1,e_3]=\lambda e_3, \ [e_2,e_3]=0.$$

For every $\lambda$, with $|\lambda|<1$, the automorphism group $\mathrm{Aut}(\h(\lambda))$ can be obtained, by direct calculation using the definition and our preferred basis $B=\{e_1,e_2,e_3\}$, as being the matrix group
\begin{equation*}
 \mathrm{Aut}(\h(\lambda)) =   \left\{
        \begin{pmatrix}
            1 & 0 & 0 \\
            a & c & 0\\
            b & 0 & d
        \end{pmatrix}: \, a,b,c,d \in \mathbb{R}, c,d\neq 0
    \right\}.
\end{equation*}
As can be seen, $\mathrm{Aut}(\h(\lambda))$ is 4-dimensional and has four connected components, with $\mathrm{Aut}(\h(\lambda))/\mathrm{Aut}^0(\h(\lambda))$ isomorphic to the Klein group $\Z/2\Z\times\Z/2\Z$. Consider the generic $3\times 3$ matrix  
\begin{equation}
m= \begin{pmatrix}
    m_1 & m_2 & m_3 \\
    m_2 & m_4 & m_5 \\
    m_3 & m_5 & m_6
\end{pmatrix},
\end{equation}
which is assumed to represent a non-degenerate symmetric bilinear form with respect the basis $B$. The image of $m$ under the automorphism $$\varphi^{-1} = \begin{pmatrix}
    1 & 0 & 0 \\
    a & c & 0 \\
    b & 0 & d
\end{pmatrix}$$ is given by $\varphi^Tm\varphi$ where $T$ denotes the matrix transpose. We observe that the restriction of $m$ to the ideal $\mathfrak{d}=\spann\{e_2,e_3\}$ (the derived subalgebra when $\lambda\neq 0$) is transformed only by the subgroup $$\{\varphi \in \Aut(\h(\lambda)):\, a=b=0\}.$$ Moreover, the (non-)degeneracy on $\mathfrak{d}$ and whether $e_2$ or $e_3$ are isotropic or not are all preserved by $\Aut(\h(\lambda))$. We thus have two cases to consider, which will include subcases.

\smallskip

Case 1: $m|_\mathfrak{d}$ is non-degenerate i.e. $m_4m_6 -m_5^2 \neq 0$.

Subcase 1.1: $e_2$ is non-isotropic.

If $m_5=0$, we have 
$$
    \begin{array}{ll}
        \mathcal{Q}_{0} = \begin{pmatrix}
    1 & 0 & 0\\
    0 & 1 & 0 \\
    0 & 0 & 1
    \end{pmatrix}, &\mathcal{Q}_{1} = \begin{pmatrix}
    1 & 0 & 0\\
    0 & 1 & 0 \\
    0 & 0 & -1
    \end{pmatrix}, \medskip \\ \mathcal{Q}_{2} = \begin{pmatrix}
    1 & 0 & 0\\
    0 & -1 & 0 \\
    0 & 0 & 1
    \end{pmatrix}  \mbox{ and } &\mathcal{Q}_{3} = \begin{pmatrix}
    -1 & 0 & 0\\
    0 & 1 & 0 \\
    0 & 0 & 1
    \end{pmatrix}.
    \end{array}
$$

If $m_5 \neq 0$, we have 
$$
\begin{array}{lll}
     \mathcal{Q}_{4,r} = \begin{pmatrix}
    1 & 0 & 0\\
    0 & 1 & 1 \\
    0 & 1 & r
    \end{pmatrix}, &   \mathcal{Q}_{5,s} = \begin{pmatrix}
    1 & 0 & 0\\
    0 & -1 & 1 \\
    0 & 1 & s
    \end{pmatrix}, &\mathcal{Q}_{6,t} = \begin{pmatrix}
    -1 & 0 & 0\\
    0 & 1 & 1 \\
    0 & 1 & t
    \end{pmatrix},   
    \end{array}
$$
 with $r\neq1, s>-1, t>1$.

Subcase 1.2: $e_2$ is isotropic and $e_3$ is not isotropic i.e. $m_4 = 0$ and $m_6\neq0$.
\begin{center}
$\mathcal{Q}_{7} = \begin{pmatrix}
    1 & 0 & 0\\
    0 & 0 & 1 \\
    0 & 1 & 1
    \end{pmatrix}$ \qquad and \qquad  $\mathcal{Q}_{8} = \begin{pmatrix}
    1 & 0 & 0\\
    0 & 0 & 1 \\
    0 & 1 & -1
    \end{pmatrix}$.
\end{center}
Subcase 1.3: $e_2$ and $e_3$ are isotropic i.e. $m_4=m_6=0$.
\begin{center}
$\mathcal{Q}_{9} = \begin{pmatrix}
    1 & 0 & 0\\
    0 & 0 & 1 \\
    0 & 1 & 0
    \end{pmatrix}$.
\end{center}

\smallskip

Case 2: $m|_\mathfrak{d}$ is degenerate i.e. $m_4m_6 -m_5^2 = 0$.

Case 2.1: $e_2$ and $e_3$ are not isotropic i.e. $m_4,m_6\neq0$.

\begin{center}
$\mathcal{Q}_{10} = \begin{pmatrix}
    0 & 0 & 1\\
    0 & 1 & 1 \\
    1 & 1 & 1
    \end{pmatrix}$.
\end{center}

Case 2.2: $e_2$ is isotropic and $e_3$ is not isotropic i.e. $m_4=0$ and $m_6\neq0$.

\begin{center}
$\mathcal{Q}_{11} = \begin{pmatrix}
    0 & 1 & 0\\
    1 & 0 & 0 \\
    0 & 0 & 1
    \end{pmatrix}$.
\end{center}

Case 2.3: $e_2$ is not isotropic and $e_3$ is isotropic i.e. $m_4\neq0$ and $m_6=0$.

\begin{center}
$\mathcal{Q}_{12} = \begin{pmatrix}
    0 & 0 & 1\\
    0 & 1 & 0 \\
    1 & 0 & 0
    \end{pmatrix}$.
\end{center}

We remark that every Riemannian metric belongs to the orbit of $\mathcal{Q}_0$ or the orbit of $\mathcal{Q}_{4,r}$, for some $r>1$. 

\begin{remark}\label{Rem:limiting} Although not our main object of study, we include here the metric normal forms for $\h(\lambda), \lambda 
=\pm 1$, since they can be easily computed and will be useful in Sec. \ref{sec:limiting}.

The automorphism group of $\h(-1)$ is given by 
\begin{equation*}
 \mathrm{Aut}(\h(-1)) =   \left\{
        \begin{pmatrix}
            1 & 0 & 0 \\
            a & c & 0\\
            b & 0 & d
        \end{pmatrix},  \begin{pmatrix}
            -1 & 0 & 0 \\
            a & 0 & c\\
            b & d & 0
        \end{pmatrix}: \, a,b,c,d \in \mathbb{R}, c,d\neq 0
    \right\}.
\end{equation*}
Since $\mathrm{Aut}(\h(\lambda)), |\lambda|<1$, is a normal subgroup of $\mathrm{Aut}(\h(-1))$, then the metric normal forms of $\h(-1)$ are a subset of the previous list and can be deduced by checking which ones are in the same orbit under the action of  $\mathrm{Aut}(\h(-1))/\mathrm{Aut}(\h(\lambda))$. We obtain
\begin{equation*}
    \mathcal{Q}_0, \, \mathcal{Q}_1, \,  \mathcal{Q}_3, \, \mathcal{Q}_{4,r}, \, \mathcal{Q}_{5,s}, \, \mathcal{Q}_{6,t},\, \mathcal{Q}_9,\, \mathcal{Q}_{10}, \, \mathcal{Q}_{11},
\end{equation*}
with $r\neq 1$, $-1<s\leq 0$ and $t>1$.

The automorphism group of $\h(1)$ is given by 
\begin{equation*}
 \mathrm{Aut}(\h(1)) =   \left\{
        \begin{pmatrix}
            1 & 0 & 0 \\
            a & b & c\\
            d & e & f
        \end{pmatrix}: \, a,b,c,d, e,f \in \mathbb{R}, bf-ce\neq 0
    \right\}.
\end{equation*}
We see that $\mathrm{Aut}(\h(1))$ is 6-dimensional and this freedom quickly allows us to conclude that we only have four metric normal forms, namely 
\begin{equation*}
    \mathcal{Q}_0,\, \mathcal{Q}_1,\, \mathcal{Q}_3,\, \mathcal{Q}_{11}.
\end{equation*}
\end{remark}

\section{Geodesic vector fields on $\h(\lambda), |\lambda|<1$}\label{Sec:geod-vf}
For the Lie algebras $\h(\lambda)$, $|\lambda|<1$, and each of the metric normal forms of Lorentzian signature of Sec. \ref{Sec:N-F}, we exhibit its corresponding Euler-Arnold geodesic field and some extra properties: existence of invariant planes and idempotents. See Table \ref{table:geodesic-fields}.

\begin{table}[H]

\begin{adjustbox}{valign=t, minipage=\linewidth, scale=0.76, left}
\begin{tabular}{*5c} \toprule 
    & \quad bilinear form \quad  &  \quad geodesic field \quad & \quad invariant planes \quad & \quad  idempotents \quad \\ \toprule
 $\mathcal{Q}_{1}$  & $\begin{pmatrix}
    1 & 0 & 0\\
    0 & 1 & 0 \\
    0 & 0 & -1
    \end{pmatrix}$  & $\begin{cases}
	\dot{x}=-y^2+\lambda z^2\\
	\dot{y}=xy\\ 
	\dot{z}=\lambda xz
\end{cases}$ & \begin{tabular}{l}
     $y=0$ \\ \\ $z=0$      
\end{tabular} &  Yes, if $\lambda\neq0$  \\ \midrule 

 $\mathcal{Q}_{2}$  & $\begin{pmatrix}
    1 & 0 & 0\\
    0 & -1 & 0 \\
    0 & 0 & 1
    \end{pmatrix}$  & $\begin{cases}
	\dot{x}=y^2-\lambda z^2\\
	\dot{y}=xy\\ 
	\dot{z}=\lambda xz
\end{cases}$ & \begin{tabular}{l}
     $y=0$ \\ \\ $z=0$      
\end{tabular} & Yes \\ \midrule  

 $\mathcal{Q}_{3}$  & $\begin{pmatrix}
    -1 & 0 & 0\\
    0 & 1 & 0 \\
    0 & 0 & 1
    \end{pmatrix}$  & $\begin{cases}
	\dot{x}=y^2+\lambda z^2\\
	\dot{y}=xy\\ 
	\dot{z}=\lambda xz
\end{cases}$ & \begin{tabular}{l}
     $y=0$ \\  \\ $z=0$       
\end{tabular} & Yes\\ \midrule  

\begin{tabular}{c}
$\mathcal{Q}_{4,r}$ \\ {\small $r<1$} \end{tabular}  &  $\begin{pmatrix}
    1 & 0 & 0\\
    0 & 1 & 1 \\
    0 & 1 & r
    \end{pmatrix} $ & $\begin{cases}
	\dot{x}=-y^2-\lambda rz^2-(\lambda+1)yz\\
	\dot{y}=\frac{x}{r-1}((r-\lambda)y+r(1-\lambda)z)\\ 
	\dot{z}=\frac{x}{r-1}((\lambda-1)y+(r\lambda-1)z)
\end{cases}$ & \begin{tabular}{l}
     $y=-z$  \\  \\$y=-rz$      
\end{tabular} & \begin{tabular}{c} Yes, if \\  $0<r<1$  \\ or \\ $\lambda\neq 0$ and $r<1$ \end{tabular} \\ \midrule  

   \begin{tabular}{c}  $\mathcal{Q}_{5,s}$ \\ {\small $s>-1$} \end{tabular}  &  $\begin{pmatrix}
    1 & 0 & 0\\
    0 & -1 & 1 \\
    0 & 1 & s
    \end{pmatrix} $ & $\begin{cases}
	\dot{x}=y^2-\lambda sz^2-(\lambda+1)yz\\
	\dot{y}=\frac{x}{s+1}((s+\lambda)y+s(\lambda-1)z)\\ 
	\dot{z}=\frac{x}{s+1}((\lambda-1)y+(s\lambda+1)z)
\end{cases}$ & \begin{tabular}{l}
     $y=z$ \\ \\ $y=-sz$      
\end{tabular} & Yes, if $s>0$  \\ \midrule  

\begin{tabular}{c}  $\mathcal{Q}_{6,t}$ \\ {\small $t>1$} \end{tabular}  &  $\begin{pmatrix}
    -1 & 0 & 0\\
    0 & 1 & 1 \\
    0 & 1 & t
    \end{pmatrix} $ & $\begin{cases}
	\dot{x}=y^2+\lambda tz^2+(\lambda+1)yz\\
	\dot{y}=\frac{x}{t-1}((t-\lambda)y+t(1-\lambda)z)\\ 
	\dot{z}=\frac{x}{t-1}((\lambda-1)y+(t\lambda-1)z)
\end{cases}$ & \begin{tabular}{l}
     $y=-z$ \\ \\ $y=-tz$      
\end{tabular} & Yes  \\ \midrule 

 $\mathcal{Q}_{7}$  & $\begin{pmatrix}
    1 & 0 & 0\\
    0 & 0 & 1 \\
    0 & 1 & 1
    \end{pmatrix}$ & $\begin{cases}
	\dot{x}=-\lambda z^2-(\lambda+1)yz\\
	\dot{y}=\lambda xy+(\lambda-1)xz\\ 
	\dot{z}=xz
\end{cases}$ & \begin{tabular}{l}
     $y=-z$\\  \\ $z=0$      
\end{tabular} & Yes \\ \midrule  
 $\mathcal{Q}_{8}$  & $\begin{pmatrix}
    1 & 0 & 0\\
    0 & 0 & 1 \\
    0 & 1 & -1
    \end{pmatrix}$ &$\begin{cases}
	\dot{x}=\lambda z^2-(\lambda+1)yz\\
	\dot{y}=\lambda xy+(1-\lambda)xz\\ 
	\dot{z}=xz
\end{cases}$ & \begin{tabular}{l}
     $y=z$  \\ \\ $z=0$      
\end{tabular} & No \\ \midrule  
   
 $\mathcal{Q}_{9}$  &  $\begin{pmatrix}
     1 & 0 & 0 \\
     0 & 0 & 1 \\ 0 & 1 & 0 \end{pmatrix}$ & $\begin{cases}
	\dot{x}=-(\lambda+1)yz\\
	\dot{y}=\lambda xy\\ 
	\dot{z}=xz
\end{cases}$ & \begin{tabular}{l}
     $y=0$ \\ \\ $z=0$      
\end{tabular} & No \\ \midrule 

 $\mathcal{Q}_{10}$  &  $\begin{pmatrix}
     0 & 0 & 1 \\
     0 & 1 & 1 \\ 1 & 1 & 1 \end{pmatrix}$ & $\begin{cases}
	\dot{x}=\lambda x^2+(\lambda-1)(y+z)x\\
	\dot{y}=y^2+\lambda z^2+xy+(\lambda+1)(x+y)z\\ 
	\dot{z}=-y^2-\lambda z^2-(\lambda x+(1+\lambda)y)z
\end{cases}$ & \begin{tabular}{l}
     $x=0$ \\ \\ $y=-z$      
\end{tabular} & Yes \\ \midrule
 $\mathcal{Q}_{11}$  & $\begin{pmatrix}
    0 & 1 & 0\\
    1 & 0 & 0 \\
    0 & 0 & 1
    \end{pmatrix}$  & $\begin{cases}
	\dot{x}=x^2\\
	\dot{y}=-\lambda z^2 -xy\\ 
	\dot{z}=\lambda xz
\end{cases}$ & \begin{tabular}{l}
     $x=0$\\  \\ $z=0$      
\end{tabular} & Yes \\ \midrule  
 $\mathcal{Q}_{12}$  &  $\begin{pmatrix}
     0 & 0 & 1 \\
     0 & 1 & 0 \\ 1 & 0 & 0 \end{pmatrix}$ & $\begin{cases}
	\dot{x}=\lambda x^2\\
	\dot{y}=xy\\ 
	\dot{z}=-y^2-\lambda xz
\end{cases}$ & \begin{tabular}{l}
     $x=0$\\  \\ $y=0$      
\end{tabular} & Yes, if $\lambda\neq0$  \\ \bottomrule   
\end{tabular}
\end{adjustbox}\ \caption{Geodesic vector fields on $\h(\lambda)$ with invariant planes and idempotents.}\label{table:geodesic-fields}
\end{table}

\begin{remark}
In Table \ref{table:geodesic-fields}, we listed two invariant planes for each geodesic vector field. We note that, for $0<|\lambda|<1$, no further invariant planes exist. However, this is not the case for $\lambda =0$, where other invariant planes do exist. For instance, for $\mathcal{Q}_3$, infinitely many invariant planes exist, namely $z=c$, for every $c\in \R$. Another example is that of $\mathcal{Q}_7$, 
where $y=-x$ is an additional invariant plane. When the system of differential equations does not look tractable, the existence of invariant planes can be an important tool in the study of the vector field dynamics, by means of a linear change of variables. See Subsec. \ref{subsec-q5s} for an example. Notice, however, that a geodesic field may not have invariant planes, as is the case of \eqref{eq:e-mu-b-1}.
\end{remark}

\bigskip

\section{(In)completeness on $\mathfrak{aff}(\R)\oplus\R$: proof of Theorem \ref{thm:aff}}\label{sec:(in)co_aff}

The aim of this section is to present the classification of Lorentzian metrics on $\h(0)=\mathfrak{aff}(\R)\oplus\R$ in terms of geodesic (in)completeness and prove Theorem \ref{thm:aff}. 

Let $\z$ be the center and $\mathfrak{d}$ be the derived subalgebra of $\mathfrak{aff}(\R)\oplus \R$. In our standard basis, the non-zero bracket relation is given by $[e_1, e_2]=e_2$, and thus $\z=\spann\{e_3\}$ and $\mathfrak{d}=\spann\{e_2\}$.  We start by observing that both $\z$ and $\mathfrak{d}$ are invariant under the action of $\Aut(\mathfrak{aff}(\R)\oplus \R)$ and that the induced action on the space of Lorentzian metrics preserves their type (spacelike, timelike or lightlike) as well as whether they are orthogonal or not. 

Inspecting Table \ref{table:geodesic-fields}, we readily check that we have the following partition of Lorentzian normal forms on $\mathfrak{aff}(\R)\oplus \R$:

\begin{itemize}
    \item $\z$ is spacelike: $\mathcal{Q}_2$, $\mathcal{Q}_3$, $\mathcal{Q}_{4,r}$ {\small($0<r<1$)}, $\mathcal{Q}_{5,s}$ {\small ($s>0$)}, $\mathcal{Q}_{6,t}$, $\mathcal{Q}_7$, $\mathcal{Q}_{10}$, $\mathcal{Q}_{11}$;
    \item $\z$ is timelike: $\mathcal{Q}_1$, $\mathcal{Q}_{4,r}$ {\small ($r<0$)}, $\mathcal{Q}_{5,s}$ {\small ($-1<s<0$)}, $\mathcal{Q}_8$;
\item $\z$ is lightlike and
\begin{itemize}
    \item[--] $\z$ and $\mathfrak{d}$ are not orthogonal: $\mathcal{Q}_{4,0}$, $\mathcal{Q}_{5,0}$, $\mathcal{Q}_9$;
    \item [--] $\z$ and $\mathfrak{d}$ are orthogonal: $\mathcal{Q}_{12}$.
\end{itemize}
\end{itemize}

Let us denote by $\mathcal{F}_k$ the geodesic field associated to the metric $\mathcal{Q}_k$, for every subscript $k$ as in Section \ref{Sec:N-F}.

\subsection{Incomplete metrics}

\subsubsection{Incomplete geodesic fields with idempotents}

The Lorentzian metrics whose geodesic fields have idempotents were already noted in Table \ref{table:geodesic-fields}. Indeed, they are the only ones for which $\z$ is spacelike. Here, we present the list of idempotents of the geodesic field for each such metric.

\smallskip

 \begin{itemize}
\item[]      $\mathcal{F}_2$:  $(1,1,0), (1,-1,0)$; 
\item[]      $\mathcal{F}_3$:  $(1,1,0), (1,-1,0)$; 
\item[]      $\mathcal{F}_{4,r}$ \small{$(0<r<1)$}:   $\left( 1, \frac{-r}{\sqrt{r-r^2}}, \frac{1}{\sqrt{r-r^2}}\right)$;
\item[]      $\mathcal{F}_{5,s}$ \small{$(s>0)$}:   $\left( 1, \frac{-s}{\sqrt{s+s^2}}, \frac{1}{\sqrt{s+s^2}}\right)$; 
\item[]      $\mathcal{F}_{6,t}$:   $\left( 1, \frac{-t}{\sqrt{t^2-t}}, \frac{1}{\sqrt{t^2-t}}\right)$; 
\item[]      $\mathcal{F}_7$: $(1,1,-1), (1,-1,1)$;
\item[]  $\mathcal{F}_{10}$: $\left(1, -\frac{1}{2}, -\frac{1}{2}\right)$;    
\item[]  $\mathcal{F}_{11}$: $(1,0,0)$. 
\end{itemize}

\subsubsection{Incomplete geodesic fields with no idempotents}

The Lorentzian metrics for which $\z$ is lightlike and not orthogonal to $\mathfrak{d}$ are $\mathcal{Q}_{4,0}, \mathcal{Q}_{5,0}$ and $\mathcal{Q}_9$. They admit, cf. Table \ref{table:geodesic-fields}, the following geodesic fields.

$$
 \begin{array}{lll}
\mathcal{F}_{4,0}:\, \begin{cases}
	\dot{x}=-y^2-yz\\
	\dot{y}=0\\ 
	\dot{z}=xy+xz
\end{cases};
&  
\mathcal{F}_{5,0}:\,\begin{cases}
	\dot{x}=y^2-yz\\
	\dot{y}=0\\ 
	\dot{z}=-xy+xz
\end{cases};
&
\mathcal{F}_{9}:\,\begin{cases}
	\dot{x}=-yz\\
	\dot{y}=0\\ 
	\dot{z}=xz
\end{cases}.
\end{array}
$$

Their energy is a first integral and, moreover, $y$ (and so $y^2$) is clearly another first integral. Therefore, the linear combinations
$$\begin{array}{lclcl}
    \mathcal{I}_{4,0} & = & \tfrac{1}{2}(x^2+y^2+2yz)+\tfrac{1}{2}y^2  & = & \tfrac{1}{2}x^2+y^2+yz \smallskip \\
     \mathcal{I}_{5,0} & = &\tfrac{1}{2}(x^2-y^2+2yz)-\tfrac{1}{2}y^2 &= &\tfrac{1}{2}x^2-y^2+yz \smallskip\\
      \mathcal{I}_{9} & = &\tfrac{1}{2}(x^2+2yz) & = & \tfrac{1}{2}x^2+yz \\
\end{array}$$
are again first integrals of $\mathcal{F}_{4,0}$, $\mathcal{F}_{5,0}$ and $\mathcal{F}_9$, respectively. By restricting the geodesic equations to the corresponding zero level sets of $\mathcal{I}_{4,0}$, $\mathcal{I}_{5,0}$ and $\mathcal{I}_9$, we obtain the differential equation $\dot{x}=\tfrac{1}{2}x^2$ whose non-trivial solutions have maximal domain of definition strictly contained in $\R$. Thus, $\mathcal{F}_{4,0}$, $\mathcal{F}_{5,0}$ and $\mathcal{F}_9$
 are incomplete vector fields.

\subsection{Complete metrics}

\subsubsection{Complete geodesic fields whose integral curves are  all bounded}  
The Lorentzian metrics for which $\z$ is timelike are $\mathcal{Q}_{1}$, $\mathcal{Q}_{4,r}$ {\small $(r<0)$}, $\mathcal{Q}_{5,s}$ {\small $(-1<s<0)$} and $\mathcal{Q}_8$. Their geodesic fields are given by
$$
\begin{array}{ll}
\mathcal{F}_1: \begin{cases}
	\dot{x}=-y^2\\
	\dot{y}=xy\\ 
	\dot{z}=0
\end{cases}
\qquad &
\mathcal{F}_{4,r}: \begin{cases}
    \dot{x}=-y^2-yz\\
	\dot{y}=\frac{rx}{r-1}(y+z)\\ 
	\dot{z}=\frac{-x}{r-1}(y+z)
\end{cases}
\medskip \\
\mathcal{F}_{5,s}: \begin{cases}
	\dot{x}=y^2-yz\\
	\dot{y}=\frac{sx}{s+1}(y-z)\\ 
	\dot{z}=\frac{x}{s+1}(z-y)
\end{cases}
& 
\mathcal{F}_8:\begin{cases}
	\dot{x}=-yz\\
	\dot{y}=xz\\ 
	\dot{z}=xz
\end{cases}
\end{array}
$$

cf. Table \ref{table:geodesic-fields}. Besides the energy, they have the following polynomial (in fact, linear) first integrals, $\mathcal{J}_{1}=z$, $\mathcal{J}_{4,r}=y+rz$, $\mathcal{J}_{5,s} = y+sz$, $\mathcal{J}_8=y-z$, respectively. We obtain the following quadratic first integrals
\small
$$\begin{array}{lclcl}
    \mathcal{I}_1 & = & (x^2+y^2-z^2)+2z^2 & = & x^2+y^2+z^2 \smallskip \\ \mathcal{I}_{4,r} & = & (x^2+y^2+rz^2+2yz)-\frac{r-2}{r(r-1)}(y+rz)^2 & = &x^2+\frac{1}{r(r-1)} ((r-1)^2y^2+(y+rz)^2) \smallskip \\
    \mathcal{I}_{5,s} & = & (x^2-y^2+sz^2+2yz)-\frac{s+2}{s(s+1)}(y+sz)^2 & = &x^2-\frac{1}{s(s+1)} ((s+1)^2y^2+(y+sz)^2) \smallskip \\
    \mathcal{I}_8 & = & (x^2+2yz-z^2)+2(y-z)^2 & = & x^2 + y^2 + (y-z)^2
\end{array}$$ \normalsize
and observe that they are all positive-definite. Thus, the geodesic fields $\mathcal{F}_1$, $\mathcal{F}_{4,r}$ {\small $(r<0)$}, $\mathcal{F}_{5,s}$ {\small $(-1<s<0)$}, $\mathcal{F}_8$ have all their integral curves bounded and are, therefore, complete.

\begin{remark} It is well known  that the existence of a timelike Killing field on a Lorentzian manifold with some boundedness conditions implies geodesic completeness, \cite[Prop. 2.1]{RomeroSanchez}. In our case, 
 the Lie algebra vector $e_3$ is central and this implies that the corresponding left-invariant vector field $E_3$ is also right-invariant. In particular, $E_3$ is a Killing field which has constant energy. When $E_3$ is timelike, completeness of our geodesic fields then follows immediately. Moreover, it is worth noting that our linear first integrals were expected. The fact that $E_3$ is Killing implies that every geodesic field has a linear first integral given by $q(e_3, -)$, where $q$ is the metric in question.           
\end{remark}

\subsubsection{Complete geodesic fields whose non-stationary integral curves are all unbounded}
The remaining Lorentzian metric in our discussion is the only case where $\z$ is lightlike and orthogonal to $\mathfrak{d}$, that is, $\mathcal{Q}_{12}$. Its geodesic field is given by
$$\mathcal{F}_{12}: \begin{cases}
    \dot{x}=0\\
    \dot{y}=xy\\
    \dot{z}=-y^2
\end{cases}$$
and we readily see that its flow can be computed explicitly. Concretely, given the initial condition $(x_o, y_o, z_o)$, the solution $\alpha$ can be written as 
$$\begin{array}{lcl}
\alpha(t) & = & (0, y_o, -y_o^2t+z_0),\, t\in \R, \,  \mbox{if } x_o=0,\\ 
\alpha(t) & = & \left(x_o, y_o \mathrm{e}^{x_o t}, \tfrac{y_o^2}{2 x_o} (1-\mathrm{e}^{2 x_o t})+z_o\right), \, t\in\R, \, \mbox{if } x_o\neq 0.
\end{array}$$
Thus, $\mathcal{F}_{12}$ is a complete geodesic field and its non-stationary integral curves are all unbounded. 

\medskip

    The proof of Theorem \ref{thm:aff} now follows from the discussion in this section.

\section{(In)completeness on $\mathfrak{h}(\lambda)$, $0<|\lambda|<1$: proof of Theorem \ref{thm:h(lambda)}} \label{sec:(in)co_h_lambda}
In this section, we classify Lorentzian metrics on the Lie algebras $\h(\lambda)$, $0<|\lambda|<1$, with respect to their geodesic (in)completeness and give the proof of Theorem \ref{thm:h(lambda)}. Our strategy here will be similar to that of Sec. \ref{sec:(in)co_aff}.

Inspecting Table \ref{table:geodesic-fields}  once more, we see that the family $\mathcal{Q}_{5,s}$, with {\small $-1<s<0$}, corresponds to the metric normal forms for which both $e_2$ and $e_3$ are timelike. On the other hand, $\mathcal{Q}_8$ is the only one for which $e_3$ is timelike and $e_2$ is lightlike.
We start by showing that every other metric normal form is incomplete and proceed with the analysis of the completeness of $\mathcal{Q}_{5,s}$, {\small $-1<s<0$}, and $\mathcal{Q}_8$.

\subsection{Incomplete metrics}

\subsubsection{Incomplete geodesic fields with idempotents}
For each geodesic field that admits idempotents, we present their full list. 

\begin{itemize}
\item[] $\mathcal{F}_1$:  $\left(\frac{1}{\lambda}, 0, \frac{1}{\lambda} \right)$, $\left(\frac{1}{\lambda}, 0, \frac{-1}{\lambda}\right)$;
 \item[]      $\mathcal{F}_2$:  $(1,1,0), (1,-1,0)$; 
\item[]      $\mathcal{F}_3$:  $(1,1,0), (1,-1,0)$, $\left(\frac{1}{\lambda}, 0, \frac{1}{\lambda} \right)$, $\left(\frac{1}{\lambda}, 0, \frac{-1}{\lambda}\right)$; 
\item[]      $\mathcal{F}_{4,r}$ \small{$(0<r<1)$}:   $\left( 1, \frac{-r}{\sqrt{r-r^2}}, \frac{1}{\sqrt{r-r^2}}\right)$, $\left( \frac{1}{\lambda}, \frac{-1}{\lambda\sqrt{1-r}}, \frac{1}{\lambda \sqrt{1-r}}\right)$;
\item[]      $\mathcal{F}_{4,r}$ \small{$(r\leq 0)$}:   $\left( \frac{1}{\lambda}, \frac{-1}{\lambda\sqrt{1-r}}, \frac{1}{\lambda \sqrt{1-r}}\right)$;  
\item[]      $\mathcal{F}_{5,s}$ \small{$(s>0)$}:   $\left( 1, \frac{-s}{\sqrt{s+s^2}}, \frac{1}{\sqrt{s+s^2}}\right)$; 
\item[]      $\mathcal{F}_{6,t}$:   $\left( 1, \frac{-t}{\sqrt{t^2-t}}, \frac{1}{\sqrt{t^2-t}}\right)$, $\left(\frac{1}{\lambda}, \frac{-1}{\lambda\sqrt{t-1}}, \frac{1}{\lambda\sqrt{t-1}} \right)$; 
\item[]      $\mathcal{F}_7$: $(1,1,-1), (1,-1,1)$;
\item[]  $\mathcal{F}_{10}$: $\left(1, -\frac{1}{2}, -\frac{1}{2}\right)$, $\left(\frac{1}{\lambda}, 0,0\right)$;    
\item[]  $\mathcal{F}_{11}$: $(1,0,0)$;
\item[] $\mathcal{F}_{12}$:  $\left(\frac{1}{\lambda}, 0,0\right)$.
\end{itemize}

\subsubsection{Incomplete geodesic fields with no idempotents}
We consider here the metrics $\mathcal{Q}_{5,0}$ and $\mathcal{Q}_9$. Starting with $\mathcal{Q}_9$, the restriction of the geodesic field to the zero level set of the energy $x^2+2yz$ gives the equation $\dot{x}=\frac{\lambda+1}{2}x^2$. Hence, $\mathcal{F}_9$ is incomplete. For $\mathcal{Q}_{5,0}$, the restriction of the geodesic field to the zero level set of the energy $e(x,y,z)=x^2-y^2+2yz$ yields the equation
$$\dot{x}=\frac{\lambda+1}{2}x^2+\frac{1-\lambda}{2}y^2,$$ which, in turn, implies the existence of incomplete integral curves of $\mathcal{F}_{5,0}$ by Lem. \ref{lem:incompleteness}, since $\frac{1-\lambda}{2}y^2$ is nonnegative for $0<|\lambda|<1$. Thus, $\mathcal{F}_{5,0}$ is incomplete. 
\subsection{Complete metrics}

\subsubsection{Complete geodesic fields whose integral curves are not all bounded}\label{subsec-q8} 
Let us take the Lorentzian metric $\mathcal{Q}_8$ and recall that its geodesic field $\mathcal{F}_8$ is given by
\begin{equation*}
    \mathcal{F}_8: \begin{cases}
        \dot{x} = \lambda z^2-(\lambda+1)yz \\
        \dot{y} = \lambda x y + (1-\lambda) x z \\
        \dot{z} = x z
    \end{cases}.
\end{equation*}
The plane $\{z=0 \}$ is invariant under the geodesic flow and the restriction to this plane yields a complete vector field with unbounded integral curves. Precisely, for an initial condition $(x_o, y_o, 0)$, the integral curve is $\alpha(t)=(x_o, y_o\mathrm{e}^{\lambda x_o t}$, 0), $t\in \R$.

Since $\{z=0\}$ is invariant, then so are the half-spaces $H^+=\{z>0\}$ and $H^-=\{z<0\}$. It can be easily checked that, besides the energy, no other quadratic first integrals exist. However, we have the following. 

\begin{lemma} 
    The map $f: H^+ \longrightarrow \R$, $(x,y,z) \longmapsto f(x,y,z)=\dfrac{y-z}{z^\lambda} = (y-z)\mathrm{e}^{-\lambda\ln (z)}$ is an invariant of the geodesic field $\mathcal{F}_8$. 
\end{lemma}

\begin{proof}
    Let $\gamma(t)=(x(t), y(t), z(t))$, with $z(t)>0$, be an integral curve of $\mathcal{F}_8$. Then
    $$\frac{d}{dt}f(\gamma(t)) = \frac{(\dot{y}(t)-\dot{z}(t))z(t)^\lambda - (y(t)-z(t))\lambda \dot{z}(t)z(t)^{\lambda-1}}{z(t)^{2\lambda}}=0.$$
Thus, $f$ is a first integral of $\mathcal{F}_8$ restricted to $H^+$.     
\end{proof}
Let $k,c \in \R$, and denote by $\mathcal{S}_{k,c}$ the following subset of $\R^3$
$$\mathcal{S}_{k,c}=\left\{(x,y,z)\in H^+:\,  x^2-z^2+2yz =k, \, \frac{y-z}{z^\lambda}=c \right\}.$$ We will show that $\mathcal{S}_{k,c}$ is bounded for every $k$ and $c$. We have that 
$$y=z+cz^\lambda \quad \mbox{ and } \quad x^2=k-z^2(1+2cz^{\lambda-1}) $$
and this immediately implies that $x$ and $y$ are bounded if $z$ is. The second equation implies that $z^2(1+2cz^{\lambda-1})\leq k$. Thus, $z$ cannot go to infinity (as $\lambda-1<0$) and is, therefore, bounded. 
We have now shown that every integral curve with initial condition $(x_o, y_o, z_o)$ such that $z_0>0$ is contained in a compact part of $H^+$ and is thus complete. Now, since $\mathcal{F}_8$ is a quadratic homogeneous vector field, the integral curves on $H^-$ are in one-to-one correspondence with the integral curves of $H^+$  via $\alpha(t)\longmapsto -\alpha(-t)$. 
This concludes the proof that $\mathcal{F}_8$ is a complete vector field.

\subsubsection{Complete geodesic fields whose integral curves are all bounded}\label{subsec-q5s}  The remaining Lorentzian metrics in our analysis are the ones in the family $\mathcal{Q}_{5,s}$, \small{$-1<s<0$}. As seen from Table \ref{table:geodesic-fields}, the geodesic field $\mathcal{F}_{5,s}$ looks somewhat complicated. However, the existence of the two invariant planes $\{y=z\}$ and $\{y=-sz\}$ hints at the following linear change of variables 
$$\xi=sz+y \quad \mbox{ and } \quad  \eta = z-y .$$
In the new coordinates $(x,\xi, \eta)$ the vector field $\mathcal{F}_{5,s}$ is written as
$$\mathcal{F}_{5,s}: \,\ \begin{cases}
	\dot{x}=\frac{1}{s+1}(-\lambda\xi^2+s\eta^2-(\lambda+1)\xi\eta)\\
	\dot{\xi}=\lambda x\xi\\ 
	\dot{\eta}=x\eta
\end{cases}$$
and, indeed, it becomes simpler. The energy is given by $e_s(x,\xi, \eta)= x^2+\frac{1}{s+1}(\xi^2+2\xi\eta-s\eta^2)$. The invariant plane $\{\eta=0\}$ is spacelike and all integral curves with initial condition $(x_o, \xi_o, \eta_o)$ with $\eta_o=0$ are bounded, and hence complete. The half-spaces $H^+=\{\eta>0\}$ and $H^-=\{\eta<0\}$ are also invariant. Moreover, we have the following partially defined first integral.

\begin{lemma} 
    The map $f: H^+ \longrightarrow \R$, $(x,\xi,\eta) \longmapsto f(x,\xi,\eta)=\dfrac{\xi}{\eta^\lambda} = \xi\,\mathrm{e}^{-\lambda\ln (\eta)}$ is an invariant of the geodesic field $\mathcal{F}_{5,s}$. 
\end{lemma}

Let $k,c \in \R$, and denote by $\mathcal{S}_{k,c}$ the following subset of $\R^3$
$$\mathcal{S}_{k,c}=\left\{(x,y,z)\in H^+:\, x^2+\frac{1}{s+1}(\xi^2+2\xi\eta-s\eta^2)=k, \, \frac{\xi}{\eta^\lambda}=c \right\}.$$ We will show that $\mathcal{S}_{k,c}$ is bounded for every $k$ and $c$. We have that 
$$\xi=c\,\eta^\lambda \quad \mbox{ and } \quad x^2=k-\frac{1}{s+1}(c^2\eta^{2\lambda}+2c\,\eta^{\lambda+1}-s\,\eta^2) $$
and this immediately implies that $x$ and $\xi$ are bounded if $\eta$ is. The second equation implies that $$-\frac{s}{s+1}\eta^2\left(1-\frac{c^2}{s}\eta^{2(\lambda-1)}-\frac{2c}{s}\eta^{\lambda-1}\right)\leq k.$$ 
It is important to observe here that not only $\lambda-1<0$ but also that $-\frac{s}{s+1}>0$ since $-1<s<0$. 

Thus, $\eta$ cannot go to infinity and is, therefore, bounded. The rest of the argument to prove that $\mathcal{F}_{5,s}$ is a complete vector field is now analogous to that of Subsec. \ref{subsec-q8}.

\medskip

This concludes the proof of Theorem \ref{thm:h(lambda)}.

\bigskip

\section{The limiting Lie algebras $\h(\lambda)$, with $\lambda=\pm 1$}\label{sec:limiting}

The geodesic fields of Table \ref{table:geodesic-fields} in Sec. \ref{Sec:geod-vf}  clearly extend to the limiting cases $\lambda = \pm 1$ (although care should be taken concerning the existence of idempotents). Together with the metric normal forms of Rem. \ref{Rem:limiting}, we can quickly solve the completeness problem for $\h(1)$ and $\h(-1)$, recovering the result of \cite{Guediri-solvable} in dimension 3 and correcting the statement in \cite[Prop. 3]{BrombergMedina}.

\subsection{Case $\lambda =1 $} The Lorentzian normal forms are $\mathcal{Q}_1, \mathcal{Q}_3$, and $\mathcal{Q}_{11}$, and they are all incomplete with idempotents.

\subsection{Case $\lambda =-1$} From Table \ref{table:geodesic-fields}, we readily check that all metric normal forms given in Rem. \ref{Rem:limiting} are incomplete with idempotents except for $\mathcal{Q}_{5,s}$ {\small ($-1<s \leq 0$)} and $\mathcal{Q}_9$. We now show that these remaining metrics are complete. 
The geodesic field associated to $\mathcal{Q}_9$ is given by 
   \begin{equation*}   
  \mathcal{F}_9: \,  \begin{cases}
	\dot{x}=0\\
	\dot{y}=-xy\\ 
	\dot{z}= xz
\end{cases}.
   \end{equation*}
The integral curve of $\mathcal{F}_9$ with initial condition $(x_o, y_o, z_o)$ is, thus, $\alpha(t)=(x_o, y_o\mathrm{e}^{-x_o t}, z_o\mathrm{e}^{x_o t})$, $t\in\R$. If $y_o\neq 0$ or $z_o\neq 0$, then $\alpha(t)$ is unbounded. 

As for the family $\mathcal{Q}_{5,s}$ {\small ($-1<s \leq 0$)}, we can easily check that, besides the energy $e(x,y,z)=x^2-y^2+2yz+sz^2$, the first integral of Subsec. \ref{subsec-q5s} extends to $\lambda=-1$ and yields, in fact, the quadratic first integral $f(x,y,z) = y^2+(s-1)yz-sz^2$. Taking $e(x,y,z)+2f(x,y,z) = x^2+(y+sz)^2-(s^2+s)z^2$, we obtain a positive-definite quadratic first integral, for all $-1<s<0$. So, all integral curves of $\mathcal{F}_{5,s}$ {\small ($-1<s<0$)} are bounded. For $s=0$, we have the geodesic field
   \begin{equation*}   
  \mathcal{F}_{5,0}: \,  \begin{cases}
	\dot{x}=y^2\\
	\dot{y}=-xy\\ 
	\dot{z}= x(z-2y)
\end{cases}
  \end{equation*}
The plane $\{y=0\}$ is invariant with integral curves $\alpha(t)=(x_o, 0, z_o \mathrm{e}^{x_o t})$, $t\in \R$.  We can assume that $y\neq 0$. The solutions for $x$ and $y$ do not depend on $z$ and, moreover, $x^2+y^2 =c$, for some $c>0$. Thus $x(t)$ and $y(t)$ are defined for all $t\in \R$. Since $y^2-yz = k$, for some $k\in \R$, then $z(t)=y(t)-\frac{k}{y(t)}$ and thus $z(t)$ is also defined for all $t\in \R$. Hence, $\mathcal{F}_{5,0}$ is complete.  

In summary, indeed a Lorentzian metric is incomplete if and only if it has an idempotent. However, not all complete metrics admit a positive-definite quadratic first integral, as is the case of $\mathcal{Q}_9$ and $\mathcal{Q}_{5,0}$.

\bigskip

\section{Incompleteness of 3-dimensional non-unimodular Lie algebras}\label{sec:3D-non-uni}

Recall that, according to the Bianchi classification \cite{Bianchi}, the list of non-unimodular 3-dimensional Lie algebras is given as follows.
$$\begin{array}{ll}
    \mathfrak{psh}: &  [e_1,e_2]=e_2,\, [e_1, e_3] = e_2+e_3;  \\
     \mathfrak{h}(\lambda), \, {\small -1<\lambda\leq 1}: & [e_1, e_2] = e_2,\, [e_1, e_3] = \lambda e_3;\\
     \mathfrak{e}(\mu), {\small \mu>0}: & [e_1, e_2]= \mu e_2+ e_3,\, [e_1, e_3] =\mu e_3-e_2. 
\end{array}$$

In this section, we wish to finalize the proof that every non-unimodular 3-dimensional Lie algebra admits an incomplete Lorentzian metric. For the Lie algebra $\h(1)$, this was done in \cite{Guediri-solvable}; for $\mathfrak{psh}$ in \cite{CFZ-partI}; and for $\h(\lambda)$, $|\lambda|<1$, in the present paper. It thus remains to show that, for every $\mu> 0$, the Lie algebra $\mathfrak{e}(\mu)$ admits at least one incomplete metric.

\begin{proposition}\label{prop:e-mu}
    Every Lie algebra $\mathfrak{e}(\mu)$, with $\mu>0$, admits an incomplete metric.
\end{proposition}

\begin{proof}
       Let us take the metric $q=\mathrm{diag}(-1,1,1)$ given in a basis $\{e_1, e_2, e_3\}$ as above,  whose energy is $e(x,y,z)=-x^2+y^2+z^2$ and the geodesic field is given by
    \begin{equation}   
  \label{eq:e-mu-b-1}
     \begin{cases}
	\dot{x}=\mu(y^2+z^2)\\
	\dot{y}=\mu xy+xz\\ 
	\dot{z}=-xy+\mu xz
\end{cases}.
   \end{equation}
We obtain the incomplete equation $\dot{x}=\mu x^2$ for null integral curves. Hence, $q$ is incomplete.
\end{proof}

We have thus proven the following (Thm. \ref{Thm:Intro-Ic-3D}).

\begin{theorem}\label{Thm: Ic-3D}
    Every non-unimodular Lie algebra of dimension 3 admits an incomplete metric.
\end{theorem}

We end this article with the following question: 
{\it Does every non-unimodular Lie algebra admit an incomplete metric or is this a particular phenomenon of dimension 3?}

\bigskip

\section*{Acknowledgments}  
{\small

\begingroup
\sloppy

The authors  acknowledge the support of CMAT (Centro de Matem\'atica da Universidade do Minho). Their research was financed by Portuguese Funds through FCT (Fundação para a Ciência e a Tecnologia, I.P.) within the projects UIDB/00013/2020, UIDP/00013/2020, UID/00013/2025 and also within the doctoral grant UI/BD/154255/2022 of the first named author.  

\endgroup
}

\bigskip

\bibliographystyle{alpha}
\bibliography{bibliography}

\bigskip

\end{document}